\documentclass[11pt,a4paper]{amsart}

\usepackage{graphics,epic}
\usepackage{amsmath,amssymb, amsthm,mathrsfs}
\usepackage[all,2cell]{xy}
\usepackage{bm}
\usepackage{color}

\newtheorem{theorem}{Theorem}[section]
\newtheorem*{theorem*}{Theorem}

\newtheorem*{conjecture*}{Conjecture}

\newtheorem{thm}[theorem]{Theorem}
\newtheorem{lem}[theorem]{Lemma}
\newtheorem{prop}[theorem]{Proposition}
\newtheorem{cor}[theorem]{Corollary}

\newtheorem{df}[theorem]{Definition}

\newcommand{\ie}{{\em i.e.}\ }

\newcommand{\opname}[1]{\operatorname{\mathsf{#1}}}

\renewcommand{\mod}{\opname{mod}\nolimits}

\newcommand{\add}{\opname{add}\nolimits}

\newcommand{\rank}{\opname{rank}\nolimits}

\newcommand{\ind}{\opname{ind}}

\newcommand{\Fac}{\opname{Fac}}
\newcommand{\APR}{\opname{APR}}
\newcommand{\lcm}{\opname{lcm}}

\newcommand{\cok}{\opname{cok}\nolimits}

\renewcommand{\ker}{\opname{ker}\nolimits}

\newcommand{\X}{\mathbb{X}}

\newcommand{\Z}{\mathbb{Z}}
\newcommand{\N}{\mathbb{N}}
\newcommand{\Q}{\mathbb{Q}}

\newcommand{\D}{\mathbb{D}}

\renewcommand{\P}{\mathbb{P}}

\newcommand{\ra}{\rightarrow}

\newcommand{\Hom}{\opname{Hom}}

\newcommand{\Ext}{\opname{Ext}}

\newcommand{\Aut}{\opname{Aut}}
\newcommand{\coh}{\opname{coh}}
\newcommand{\vect}{\opname{vect}}

\newcommand{\SL}{\opname{SL}}

\newcommand{\Pic}{\opname{Pic}}

\newcommand{\cg}{{\mathcal G}}

\newcommand{\co}{{\mathcal O}}

\newcommand{\ct}{{\mathcal T}}

\newcommand{\mc}{\mathcal{C}}
\newcommand{\md}{\mathcal{D}}

\newcommand{\mh}{\mathcal{H}}

\newcommand{\ma}{\mathcal{A}}

\newcommand{\mt}{\mathcal{T}}

\newcommand{\ms}{\mathcal{S}}

\newcommand{\mw}{\mathcal{W}}

\newcommand{\tauni}{\tau^{-1}}

\newcommand{\homh}{\Hom_{\mathcal{H}}}

\newcommand{\xra}{\xrightarrow}

\newcommand{\lam}{\bm{\lambda}}

\begin{document}

\title[tilting graph of tubular type]{Mutation of tilting bundles of tubular type}

\author{Shengfei Geng}
\address{Shengfei Geng\\Department of Mathematics\\SiChuan University\\610064 Chengdu\\P.R.China}
\email{genshengfei@scu.edu.cn}

\maketitle

\begin{abstract}

Let $\X$ be a  weighted projective line of tubular type and $\coh\X$ the category of coherent sheaves on  $\X$. The main purpose of this note is to show that the subgraph of the tilting graph consisting of all basic tilting bundles of $\coh\X$ is connected. This   yields an alternative proof for the connectedness of the tilting graph of $\coh\X$.
Our approach leads to the investigation of the change of slopes of a tilting sheaf in $\coh\X$ under (co-)APR mutations, which may be of independent interest.

\end{abstract}

\section{Introduction}~\label{introduction}
Tilting quiver of a hereditary abelian category was introduced by Happel and Unger~\cite{HU}, which encodes information of mutations of tilting objects. Its underlying graph is called the tilting graph. The connecetedness of tilting graphs for hereditary catetgories is a basic problem in tilting theory, which has important application in the categorification theory of cluster algebras~({\it cf.}~\cite{BMRRT, CK,BG} for instance). Among others, building on the connectedness of tilting graphs established by Happel and Unger~\cite{HU}, Buan et al.~\cite{BMRRT} established the connectedness of cluster-tilting graphs of cluster categories, which implies that there is a one-to-one correspondence between indecomposable rigid objects of a cluster category and cluster variables of the associated cluster algebra. The connectedness of (cluster)-tilting graphs
also play a fundamental role in the representation-theoretic approach to the denominator conjecture in cluster algebras ({\it cf.}~\cite{GP, FG17}).

 Let $\mathcal{H}$ be a connected hereditary abelian category over an algebraically closed field $k$ and $\mathcal{G}(\mathcal{T}_\mathcal{H})$ the tilting graph of $\mathcal{H}$. According to Happel's classification~\cite{H}, each connected hereditary abelian category with tilting objects over $k$ is derived equivalent to $\mod kQ$ for a finite acyclic quiver $Q$ or the category $\coh\X$ of coherent sheaves on a weighted projective line $\X$.
If $\mathcal{H}$ is derived to $\mod kQ$ for an acyclic quiver $Q$ which is not of wild type, Happel and Unger~\cite{HU} obtained an explicitly characterization of the connectedness of  $\mathcal{G}(\mathcal{T}_\mathcal{H})$. Moreover, they also proved that $\mathcal{G}(\mathcal{T}_\mathcal{H})$ is connected provided that $Q$ is of wild type and $\mathcal{H}$ does not contain non zero projective objects.
Weighted projective lines are classified as domestic type, tubular type and wild type according to their genus. It is known that $\coh\X$ of domestic type is derived to  $\mod kQ$ for an acyclic quiver $Q$ of tame type, while $\coh\X$ is never derived to a hereditary algebra when $\X$ is of tubular or wild type. It was conjectured by Happel and Unger~\cite{HU} that the tilting graph $\mathcal{G}(\mathcal{T}_\X) $ of $\coh\X$ is connected for a tubular and wild weighted projective line $\X$. This has been confirmed by Barot, Kussin and Lenzing~\cite{BKL} for tubular type and by Fu and Geng~\cite{FG18} for wild type.

Let $\X$ be a weighted projective line over $k$. Denote by $\vect\X$ the full subcategory of $\coh\X$ consisting of vector bundles. Recently, the subcategory $\vect\X$ has obtained a lot of attention ({\it cf.}~\cite{KLM1, KLM2, L, CLR}). A tilting object $T$ of $\coh\X$ lying in $\vect\X$ is called a {\it tilting bundle}. Denote  by   $\cg(\mt^v_{\X})$ the subgraph of the tilting graph $\cg(\mt_{\X})$  consisting of all basic tilting bundles in $\coh\X$. We are interested in  whether the subgraph  $\cg(\mt^v_{\X})$ is also connected.
  When $\X$ is of domestic type, the Auslander-Reiten quiver of $\vect\X$ is equivalent to $\Z Q$ with $Q$ is an extended Dynkin quiver. By suitable $\APR$ mutations and co-$\APR$ mutations, one can get  the subgraph  $\cg(\mt^v_{\X})$  is connected ({\it cf.}~\cite{BKL}).   When $\X$ is of tubular type, the Auslander-Reiten quiver of $\vect\X$  is consisted of standard tubes.    In this note,  by investigating some properties of mutations on tilting objects, 
  we find that $\cg(\mt^v_{\X})$  is also connected,  \ie we get the following main results:
  \begin{theorem}(see Theorem~\ref{main result})\label{main thm1}
Let $\X$ be a weighted projective line of tubular type.  The subgraph $\cg(\mt^v_{\X})$ of the tilting graph $\cg(\mt_\X)$ consisting of all the basic tilting bundles  in $\coh\X$   is  connected.
\end{theorem}
As an easy consequence, Theorem~\ref{main thm1} yields an alternative proof for the connectedness of the tilting graph $\cg(\mt_\X)$.
In order to establish the main result, we introduce APR muation and co-APR mutation on tilting objects in $\coh\X$. For an object $E\in\coh\X$, denote by $\mu E$ the slope of $E$. The change of slopes of indecomposable direct summands of a tilting object under mutation plays a center role, which is of independent interest.
\begin{thm}(see Theorem~\ref{slop change under APR and coAPR mutation})\label{slop change under APR and coAPR mutation2}
Let $\X$ be a weighted projective line of tubular type, $T=\bigoplus_{i=1}^n T_i$ be a tilting sheaf in $\coh\X$ with $\mu T_i \leq \mu T_j$ for $1\leq i<j\leq n$. For $1\leq k\leq n$, let  $\mu_{T_k}(T)=T_k^*\oplus (\bigoplus_{i\neq k} T_i)$ be the mutation of $T$ at $T_k$. Then
\begin{itemize}
 \item[(1)] $\mu T_1\leq \mu T_1^*\leq \mu T_n$ and $\mu T_1\leq \mu T_n^*\leq \mu T_n$.

\item[(2)] If $T_k$ is a first object of $T$, then $\mu T_k\leq \mu T_k^*\leq \mu T_n$.
\item[(3)] If $T_k$ is a last object of $T$, then $\mu T_k\ge \mu T_k^*\ge\mu T_1$.

\end{itemize}
\end{thm}

The paper is organized as follows. In Section~\ref{S:basic-notions}, we recall some  definitions and properties of coherent sheaves on a weighted projective line.
We then focus on weighted projective lines of tubular type. For a weighted projective line $\X$ of tubular type, we discuss the tilting objects in $\coh\X$, the automorphisms of bounded derived category $\md^b(\X)$ of $\coh\X$ and the perpendicular category of some quasi-simple rigid object in $\coh\X$ in Section~\ref{properties of tubular type}. In Section~\ref{bundle-mutation}, for some special direct summands of  a tilting object in $\coh\X$,  we  discuss how the slopes changed under mutations. In particular, Theorem~\ref{slop change under APR and coAPR mutation2} is proved. We introduce APR mutation and co-APR mutation for tilting objects and proved that the APR mutation and co-APR mutation on a tilting bundle are tilting bundles. The proof of Theorem~\ref{main thm1} is presented in Section ~\ref{connectedness of tilting bundles}.

\noindent{\bf Notation}.
For a weighted projective line $\X$ with weight sequence $(p_1,p_2,\cdots,p_t)$, we denote by   $\ind\X$ the set of all indecomposable objects in $\coh\X$,  $\ind \md^b(\X)$ the set of all indecomposable objects in $\md^b(\X)$.  For an indecomposable object $X$, we denote by $\mu X$ the slope of $X$.  Denote by $p=\lcm(p_1,\cdots,p_t)$ the least common multiple of $p_1,p_2,\ldots,p_t$. For an indecomposable object $Z$ in a standard tube, we denote by $\mt_Z$ the tube where $Z$ lies in, $\mw_Z$ the wing determined by Z, $q.l. Z$ the quasi-length of $Z$ in $\mt_Z$.  

\section{ Preliminary}~\label{S:basic-notions}

\subsection{ Coherent  sheaves associated to a weighted projective line}

Let $\X=\X(\mathbf{p},\lam)$ be a weighted projective line attached to a weight sequence $\mathbf{p}=(p_1,\dots,p_t)$
 of positive integers $p_i$ and
a  parameter sequence  $\lam=(\lambda_1\dots,\lambda_t)$  of pairwise distinct elements of $\P_1(k)$.
Denote by $L(\mathbf{p})$  the rank one additive group
$$L(\mathbf{p}):=\langle\vec{x}_1,\cdots,\vec{x}_t|p_1\vec{x}_1=\cdots=p_t\vec{x}_t=:\vec{c}\rangle$$
and  $S(\mathbf{p},\lambda)$  the $L(\mathbf{p})$-graded commutative algebra
$$S(\mathbf{p},\lambda)=k[u,v,x_1,\cdots,x_t]/(x_i^{p_i}-\lambda_i'u-\lambda_i^{''}v|i=1,\cdots,t)$$
where $\deg x_i=\vec{x}_i$ and $\lambda_i=[\lambda_i':\lambda_i^{''}]\in\P^1$.  In~\cite{GL1}, Geigle-Lenzing have associated to each weighted projective line $\X(\mathbf{p},\lam)$ a category $\coh\X$ of coherent sheaves on $\X$, which is
  the quotient category of
finitely generated $L(\mathbf{p})$-graded $S(\mathbf{p},\lam)$-modules by the Serre subcategory of finite length modules.  Geigle-Lenzing showed that  $\coh\X$ is a connected,   $k$-linear hereditary abelian  category with finite dimensional $\Hom$ and $\Ext$ spaces. 
 The free module $S(\mathbf{p},\lam)$ yields a structure sheaf $\co$, and shifting the grading gives twists $E(\vec{x})$ for any  $E \in\coh\X$ and $\vec{x} \in L(\mathbf{p}).$ Moreover, they showed that, putting $\vec{w} := \Sigma_{i=1}^t(\vec{c}-\vec{x}_i)-2\vec{c}$, $\coh\X$ has Serre dual in the form $\D\Ext^1_{\X}(E,F)=\Hom_{\X}(F,\tau E)$ for all $E,F\in \coh\X$ where $\tau E=E(\vec{w})$.

 Denote by $\vect\X$ the full subcategory of $\coh\X$ consisting of vector bundles, i.e. torsion-free sheaves, and by $\coh_0\X$  the full subcategory consisting of sheaves of finite length, i.e. torsion sheaves.
 Each coherent sheaf is a direct sum of a vector bundle and a finite length sheaf. Each vector bundle has a finite filtration by line bundles and there is no nonzero morphism from $\coh_0\X$ to $\vect \X$.


There are ordinary simple sheaves  $S_{\mu}$( $\mu\in \P^1\backslash \lam$ ) and exceptional simple sheaves $S_{i,j}$($1\leq i\leq t, 1\leq j\leq p_i$). It is known that $S(\vec{x})=S,\Hom_{\X}(\co,S)\neq 0$ 
for each ordinary simple sheaf $S$. For the exceptional simple sheaves, 
for each $1\leq i\leq t$, there is exactly one integer $j(1\leq j\leq p_i)$ such that $\Hom_{\X}(\co,S_{i,j})\neq 0$.

Let
 \[
p_{\lam}: \P^1\ra \N,\mu\mapsto \begin{cases}p_i&if\ \mu=\lambda_i\ for\ some\ i;\\1&else.\end{cases}
\]
be the weight function.

\begin{lem}\cite{GL1}\label{shape of coh_0 X}
The category $\coh_0\X $ of sheaves of finite length is an exact abelian, uniserial subcategory of $\coh\X$ which is stable under Auslander-Reiten transation.
The components of the Auslander-Reiten quiver of $\coh_0\X$ form a family of  pairwise orthogonal standard tubes $(\mt_{\mu})_{\mu\in\P^1}$ with rank
 $rk \mt_{\mu} = p_{\lam}(\mu)$.
\end{lem}



\subsection{Classifications of weighted projective line} The genus $g_{\X}$ of a weighted projective line $\X$ is defined by $$g_{\X}=1+\frac{1}{2}((t-2)p-\Sigma_{i=1}^tp/p_i).$$
A weighted projective line of genus $g_\X<1$($g_{\X}=1$, resp. $g_{\X}>1$) will be called of {\it domestic}({\it tubular}, resp. {\it wild}) type.

The domestic weight types are, up to permutation, $(q)$ with $q\ge 1$, $(q_1,q_2)$ with $q_1,q_2\ge 2$, $(2,2,n)$ with $n\ge 2$, $(2,3,3),(2,3,4),(2,3,5)$,  whereas the tubular weight types are, up to permutation, $(2,2,2,2),(3,3,3),(2,4,4)$ and $(2,3,6)$.

\subsection{Slopes of coherent sheaves} Denote by $K_0(\X)$ the the Grothendieck group of $\coh\X$, it is a free abelian group of rank $n=2+\Sigma_{i=1}^t(p_i-1)$. Recall from ~\cite{GL1}, there are two $\Z$-linear forms, rk and $\deg$ on $K_0(\X)$, called the {\it rank} and {\it degree}. In particular,  for any $\vec{x}\in L(\bf{p})$,
 $$\rank \co(\vec{x})=1, \deg\co(\vec{x})=\delta(\vec{x})$$ with $\delta:L(\mathbf{p})\ra \Z$ defined by $\delta(\vec{x}_i)=p/p_i$ where  $p=\lcm(p_1,\cdots,p_t)$ is the least common multiple of $\bf{p}$. 
For each object $E\in\coh\X$, the  {\it slope} of $E$ is defined as $$\mu E:=\deg E/\rank E{\in \mathbb{Q}\cup \{\infty\}.}$$ 
{It is obvious that $ \mu (\co(\vec{x}))=\delta(\vec{x})$.} By \cite[Lemma 2.5]{L},  for each vector bundle $E\in \coh\X$ and  $\vec{x}\in L(\bf{p})$, we have $$\mu (E(\vec{x}))=\mu E+\delta(\vec{x}).$$

An indecomposable object $E\in\coh\X$ is called {\it semistable} if for each non-trivial subbundle $E'$ of $E$, we have $\mu E'\leq   \mu E$.
 For each $q\in\Q\cup\{\infty\}$, we denote by $\mc_q$ the full subcategory of $\coh\X$ consisting of all semistable coherent sheaves  of slope $q$. In particular,   $\mc_{\infty}$ is the full subcategory consisting of the objects of finite length. For $q,q'\in\Q\cup\{\infty\}$ such that $q<q'$,  according to \cite[Proposition 5.2]{GL1}, we know that
$\Hom_{\X}(\mc_q',\mc_{q})=0$. Hence, for any objects $X\in \mc_{q}, Y\in \mc_{q'}$, we have $\Ext^1_{\X}(X,Y)\cong\D\Hom_{\X}(Y,\tau X)=0$ since $\tau X$ is still in $\mc_{q}$.

\subsection{Tilting sheaf and mutation}
A sheaf $M\in\coh\X$ is called {\it rigid} if $\Ext^1_{\X}(M,M)=0$.  A sheaf  $T\in\coh\X$ is  a {\it tilting sheaf}  if $T$ is rigid and for any object $X\in\coh\X$ the condition $\Hom_{\X}(T,X)=0=\Ext_{\X}^1(T,X)$ implies that $X=0$. Moreover, if $T$ is a vector bundle, \ie $T$ has no direct summand of finite length, then we call $T$ a {\it tilting bundle}.
It is known the canonical tilting sheaf  $T_{can}=\bigoplus\limits_{0\leq \vec{x}\leq \vec{c}}\co(\vec{x})$  is a tilting bundle in $\coh\X$.

 Note that the number of (pairwise non-isomorphic) indecomposable direct summands of any tilting sheaf equals the rank $n=2+\Sigma_{i=1}^t(p_i-1)$ of the Grothendieck group $K_0(\X)$.
A rigid sheaf with $n-1$ (pairwise non-isomorphic) indecomposable direct summands is called an {\it almost complete tilting sheaf}. An indecomposable sheaf $E'$ is called a {\it complement} of the almost complete tilting sheaf $E$ if $E \oplus E'$ is a tilting sheaf.

\begin{df}
Let $T=\overline{T}\oplus X$ and $T'=\overline{T}\oplus X^*$ be two tilting objects in $\coh\X$ where $X$ and $X^*$ are indecomposable. We call $T'$ is the mutation of $T$ at $X$, denoted by $T'=\mu_X(T)$.
\end{df}

By \cite{Hu}, each almost tilting sheaf in $\coh\X$ have precisely two complements.
{In particular, for a given tilting sheaf $T$, we may mutate $T$ at any indecomposable direct summand of $T$ to obtain a new tilting sheaf.}
 Moreover, we have
\begin{lem}\cite[Proposition 2.6, 2.8]{Hu}\label{Hu}
Let $T=T_k\oplus \overline{T}$ be a tilting sheaf in $\coh\X$ with $T_k$ an indecomposable direct summand. Let $T_k\xra{f_k} B$ {a} minimal left $\add \overline{T}$-approximation and $B'\xra{g_k} T_k$ {a} minimal right  $\add \overline{T}$-approximation.
 Then
 \begin{itemize}
 \item[(1)] $f_k$ (resp. $g_k$) is either mono or epi.
  \item[(2)] $f_k $ is  mono (resp. epi) iff $g_k$ is mono (resp. epi).
  \item[(3)] let $T_k^*=\ker g_k\oplus \cok f_k$, then $T_k^*\oplus\overline{T}$ is again a tilting sheaf in $\coh\X$.
 \end{itemize}
\end{lem}
{As a direct consequence, we obtain}
\begin{lem}\label{induce}
Let $\overline{T}$ be an almost tilting sheaf in $\coh\X$, $T_k$ and $T_k^*$ are two complements of $\overline{T}$. Then either there is an exact sequence  $0\ra T_k\xra{f} B\xra{} T_k^*\ra 0$ or an exact sequence $0\ra T_k^*\ra B'\xra{g} T_k\ra 0$, where $f$ is the minimal left $\add \overline{T}$-approximation and $g$ is the minimal right $\add \overline{T}$-approximation.
Moreover,  if $\mu (T_k)<\mu (T_k^*)$,   the exact sequence $0\ra T_k\ra B\xra{g} T_k^*\ra 0$   holds.
\end{lem}

 \section{Tubular  type}\label{properties of tubular type}
 Let $\X$ be a tubular weighted projective line with weight sequence $(p_1,\cdots,p_t)$, set $p:=\lcm(p_1,\cdots,p_t)$. 

\subsection{Shape of tubular hereditary category} Let $\X$ be a tubular weighted projective line. By~\cite{GL1}, each indecomposable coherent sheaf in $\coh\X$ is semistable. Moreover, we have
\begin{thm}~\cite[Theorem 5.6]{GL1}
For each $q\in\Q\cup\{\infty\}$,  $\mc_q$ is equivalent to $\coh_0\X$, \ie  {the} components of the Auslander-Reiten quiver of $\mc_q$ form a family of  pairwise orthogonal standard tubes $(\mt_{\mu})_{\mu\in\P^1}$ with rank
 $rk \mt_{\mu} = p_{\lambda}(\mu)$.
\end{thm}



  In the following of this note,  for an indecomposable object $Z\in\coh\X$,  we denote by $\mt_Z$ the tube where $Z$ lies in, $\mw_Z$ the {\it wing} determined by $Z$,   $q.l. Z$ the {\it quasi-length} of $Z$.  If $q.l. Z=1$, we say that $Z$  is {\it quasi-simple}. We call the smallest positive integer $d$ such that $\tau^d Z=Z$ the {\it $\tau$-period} of $Z$.
{A tube with rank one is called a {\it homogeneous tube}.}

{It is known that every line bundle is quasi-simpe with integer slope and the $\tau$-period of a line bundle is $p$.} Let $X$ be an indecomposable object of $\coh\X$ lying in a  tube  with rank $r$, it is clear that
$X$ is rigid if and only if $q.l. X\leq r-1$.


\subsection{Riemann-Roch theorem}
For a weighted projective line $\X$ of arbitrary type, the {\it Euler form} on $K_0(\X)$ is given on  sheaves by $\chi(X,Y)=\dim_k\Hom_{\X}(X,Y)-\dim_k\Ext_{\X}^1(X,Y)$. The weighted form of {\it Riemann-Roch theorem} (\cite[2.9]{GL1}, \cite[3.1.14]{M2}) states that 
$$\overline{\chi}(X,Y)=rk(X)rk(Y)(p(1-g_\X))+(\mu Y-\mu X)$$
where $\overline{\chi}(X,Y)=\Sigma_{j=0}^{p-1}\chi(\tau^j X,Y)$ is the averaged Euler form.
In particular, if $\X$ is of tubular type, then  $$\overline{\chi}(X,Y)=\mu Y-\mu X.$$
Hence, one can get:
\begin{lem}\cite[Proposition 3.2]{L}\label{cor got by Riemann-Roch  theorem}
Let $\X$ be a tubular weighted projective line. For $q,q'\in\Q\cup \infty$ and $q\neq q'$, $\Hom_{\X}(\mc_q,\mc_{q'})\neq 0$ if and only if $q<q'$.
\end{lem}
\begin{proof}
If $q>q'$, it is  known that $\Hom_{\X}(\mc_q,\mc_{q'})= 0$. Hence, if 
 $\Hom_{\X}(\mc_q,\mc_{q'})\neq 0$ and $q\neq q'$, we must have $q<q'$.

Now suppose that $q<q'$.
Let $X\in\mc_q,Y\in\mc_{q'}$ be two indecomposable sheaves. Then we have $$\overline{\chi}(X,Y)=\mu Y-\mu X=q'-q\neq 0.$$
Since for any integer $j$, $\tau^j X\in\mc_q$,  $\Ext^1_{\X}(\tau^j X,Y)\cong\D\Hom_{\X}(Y,\tau^{j+1}X)=0$, then
\begin{eqnarray*}0\neq \overline{\chi}(X,Y)=\Sigma_{j=0}^{p-1}\chi(\tau^j X,Y)&=&\Sigma_{j=0}^{p-1}(\dim_k\Hom_{\X}(\tau^j X,Y)-\dim_k\Ext^1_{\X}(\tau^j X,Y))\\
&=&\Sigma_{j=0}^{p-1}\dim_k\Hom_{\X}(\tau^j X,Y).
\end{eqnarray*}
 Therefore, $\Hom_{\X}(\mc_q,\mc_{q'})\neq 0$ since $\tau^j X\in\mc_q$ for any $j$.
\end{proof}

It is easy to get that $\overline{\chi}(X,Y)$ is also equal to $\Sigma_{j=0}^{p-1}\chi(X,\tau^j Y)$. Hence, if $X$ or $Y$  lies in a homogenous tube, we have $$\mu Y-\mu X=\overline{\chi}(X,Y)=p\chi(X,Y)=p(\dim_k\Hom_{\X}(X,Y)-\Ext_{\X}^1(X,Y)).$$ Then it is easy to get
 \begin{lem}\label{mor from homogenous tube}
Let $\X$ be a tubular weighted projective line,  $X,Y\in\ind\X$, where $Y$ lies in a homogenous tube.
If $\mu X<\mu Y$, then $\Hom_{\X}(X,Y)\neq 0$,
while if $\mu X>\mu Y$, then $\Hom_{\X}(Y,X)\neq 0$.
 \end{lem}





\subsection{Some basic properities of tilting objects in tubular cases}
Let $\X$ be a tubular weighted projective line.
Similar as the properties of wings in the situation of modules \cite[Chapter X, Lemma 2.8, Lemma 2.9]{SS1}({\it cf.}~\cite[Lemma 8.3.4]{M2} for wild type), one can get
\begin{lem}\label{mor}
Let $X,Z$ be two indecomposable object in $\coh\X$ and  $X\notin \mw_Z$.  If $\Hom_{\X}(X,Z)=0$, then for any object $Y\in \mw_Z$, $\Hom_{\X}(X,Y)=0$. {Dually}, if $\Hom_{\X}(Z,X)=0$, then for any object $Y\in \mw_Z$, $\Hom_{\X}(Y,X)=0$.
\end{lem}

According Lemma~\ref{mor}, one can get
\begin{lem}\label{tilting in wing}
Let $T$ be a tilting sheaf in $\coh\X$, $Z$ be an indecomposable  direct summand of $T$. Denote by $N$  the direct sum of all the summands of $T$ lying in the wing $\mw_Z$ and $T=T'\oplus N$. Then
 \begin{itemize}
\item [(1)] $N$ is  a tilting object in $\mw_Z$.
\item[(2)]   for any  tilting object $N'$ in $\mw_Z$, $T'\oplus N'$ is a tilting sheaf in $\coh\X$.
\item[(3)]  $T$  has a quasi-simple direct summand  lying in $\mw_Z$.
\end{itemize}
\end{lem}
\begin{proof} For $(1)$, it is easy to get that $N$ is partial tilting  in $\mw_Z$. Choose  an object $M\in\mw_Z$ such that $N\oplus M$ is tilting in $\mw_Z$. Since $\homh(T',\tau Z)=\D\Ext_{\X}^1(Z,T')=0$, by Lemma~\ref{mor}, we have $\Ext_{\X}^1(M,T')=\D\homh(T',\tau M)=0$ since $\tau M\in\mw_{\tau Z}$. Similarly, we have $\Ext_{\X}^1(T',M)=0$, then we get a partial tilting object  $T'\oplus N\oplus M=T\oplus M$. Since $T$ is tilting, then $M\in\add T$. By $M\in\mw_Z$, $M\in\add N$. Therefore, $N$ is a tilting object in $\mw_Z$.

For $(2)$, similarly as $(1)$, we have $T'\oplus N'$ is  partial tilting. Since $|N|=|N'|$, $|T'\oplus N'|=|T'|+|N'|=|T'|+|N|=|T'\oplus N|=|T|$. Hence, $T'\oplus N'$  is a tilting sheaf in $\coh\X$.

For $(3)$, suppose  $q.l. Z=l$.  It is known that the wing $\mw_Z$ is equivalent to  $\mod A_l$, where $A_l$ is the hereditary algebra given as the path algebra of a quiver of Dynkin type $A_l$ with linear orientation. Hence $N$ can be seen as a tilting module in $\mod A_l$. It is known that each tilting $A_l$-module has at least one simple $A_l$-module, so $N$ has a  quasi-simple direct summand  lying in $\mw_Z$. Then we can get the result.
\end{proof}

\begin{lem}\cite[Proposition 4.2]{M1}\label{exist p periods}
Each tilting sheaf $T$ in $\coh\X$ has an indecomposable direct summand whose $\tau$-period equals $p.$
\end{lem}
By Lemma~\ref{exist p periods} and Lemma~\ref{tilting in wing}, one can get
\begin{lem}\label{exist p periods of quasi-simple object}
Each tilting sheaf $T$ in $\coh\X$ has a quasi-simple direct summand whose $\tau$-period equals $p.$
\end{lem}

For a line bundle, we have
\begin{lem}~\cite[Lemma 4.1]{M1}\label{line bundle belong to module canonical algebras}
Let $T$ be a tilting bundle in $\coh\X$. Then for any line bundle $L$,  either $\Ext^1_{\X}(T,L)=0$ or $\Hom_{\X}(T,L)=0$.
\end{lem}

\subsection{Some automorphisms of $\md^b(\X)$} Let $\X$ be a  weighted projective line. Denote the finite bounded derived category of $\coh\X$ by $\md^b(\X)$, the translation functor of $\md^b(\X)$ by $[1]$.  
Since $\coh\X$ is hereditary, every $E\in\ind(\md^b(\X))$ {belongs} to $\coh\X[l]$ for some integer $l$. Hence each self-equivalence of $\coh\X$ can induce an automorphism of $\md^b(\X)$  naturally.   Denote by $\Pic(\X)$ the Picard group which can be identified with the grading group $L(\mathbf{p})$, 
$\Pic_0(\X)$ be the subgroup of $\Pic{\X}$ consisting of all degree-preserving shifts,
$\Aut(\coh\X)$ the automorphism group  of  $\coh\X$. Define the automorphism group $\Aut \X$ of $\X$ as the subgroup of $\Aut(\coh\X)$ of automorphisms $F$ fixing the structure sheaf, \ie satisfying $F(\co)=\co$.

 For an  object $E\in\coh\X$ and any $i\in\Z$, we say $\mu (E[i])=\mu E$. An automorphism $F$ of $\md^b(\X)$ is called  {\it slope-preserving} if $\mu (F(E))=\mu E$ for any $E\in\ind \md^b(\X)$.
Denote by $\Aut_{\mu}(\coh\X)$ and $\Aut_{\mu}(\md^b(\X))$ of slope-preserving automorphism of $\coh\X$ and $\md^b(\X)$.

\begin{lem}\cite[Proposition 4.1, Proposition 4.4]{LM2}\label{slope-preserving automorphism}
Let $\X$ be a weighted projective line, then $${\Aut_{\mu}(\md^b(\X))=\langle [1]\rangle\times \Aut_{\mu}(\coh\X)},\ \ \  \Aut_{\mu}(\coh\X)=\Pic_0\X\ltimes \Aut(\X),$$
where $\langle [1]\rangle\cong \Z$ is the subgroup generated by the translation functor of $\md^b(\X)$.
\end{lem}

\subsubsection{Tubular case} Assume that $\X$ is a  weighted projective line of tubular type. In this {subsection}, we are going to discuss some automorphisms of $\md^b(\X)$. {For more details, we refer to \cite{LM1,LM2,M2}.}

\begin{thm}\cite[Theorem 5.2.6]{M2}\label{telescopic functor}
For each $q,q'\in \Q\cup \{\infty\}$, there is an automorphism $\Phi_{q',q}$ of $\md^b(\X)$ such that $\mc_q$ is mapped to $\mc_{q'}$. Moreover, this functors satisfy $\Phi_{q^{''},q}=\Phi_{q^{''},q'}\Phi_{q',q}$ and $\Phi_{q,q}=Id$.
\end{thm}

Let $B_3$ be the braid group on three strands, \ie a group  with generators $\sigma_1,\sigma_2$ and with relations $\sigma_1\sigma_2\sigma_1=\sigma_2\sigma_1\sigma_2$. Then we have 
\begin{thm}\cite[Theorem 6.3]{LM2}\label{structure of automorphism}
 $\Aut(\md^b(\X))\cong (\Pic_0(\X)\ltimes\Aut(\X))\ltimes B_3.$
\end{thm}

By \cite{LM2, M2}, there are automorphisms $\sigma_1,\sigma_2$ of $\md^b(\X)$ such that  $$\mu (\sigma_1(E))=\mu E+1\ \ \  and \ \ \ \mu (\sigma_2(E))=\frac{\mu E}{1+\mu E}$$  for any  $E\in\ind \md^b(\X)$ and the subgroup generated by $\sigma_1,\sigma_2$ is isomorphic to $B_3$. There are also an group epimorphism $f:B_3\ra \SL(2,\Z)$ defined by $f(\sigma_1)=\left(\begin{array}{cc}1&1 \\0&1\end{array}\right) $ and $f(\sigma_2)=\left(\begin{array}{cc}1&0 \\1&1\end{array}\right).$ {According to} Theorem~\ref{structure of automorphism}, one can get the following {lemma} easily.

\begin{lem}\label{LM2}
 For any linear fractional transformation $\varphi_{w}(x)=\frac{cx+a}{dx+b}$
with  $w=\left(\begin{array}{cc}c&a \\d&b\end{array}\right)\in\SL(2,\Z)$,  there is an automorphism $\overline{\varphi}_w$ of $\md^b({\X})$,  such that $\overline{\varphi}_w$ coincides with the action of $\varphi_w$ on slopes, \ie for any  $E\in\ind \md^b(\X)$, $\mu (\overline{\varphi}_w(E))=\varphi_w(\mu E)$.
\end{lem}

Dually, by Lemma~\ref{slope-preserving automorphism} and Theorem~\ref{structure of automorphism}, one can get: 
\begin{lem}\cite{LM2}\label{LM2-1}
Let $\overline{\varphi}$ be an automorphism  of $\md^b({\X})$, then there exists a  linear fractional transformation $\varphi(x)=\frac{cx+a}{dx+b}$ with  $\left(\begin{array}{cc}c&a \\d&b\end{array}\right)\in\SL(2,\Z)$,  such that the action of  $\varphi$ on the slopes  coincides with $\overline{\varphi}$. 
\end{lem}
Using Lemma~\ref{slope-preserving automorphism} and Theorem~\ref{telescopic functor}, it is easy to get
\begin{lem}\label{slop-preserving functor}
For any two quasi-simple rigid objects $X,Y\in\coh \X$ with same $\tau$-period and slope, there is a  slope-preserving  automorphism $F$ of $\md^b(\X)$ such that $F(X)=Y$. 
\end{lem}

 {According to} Lemma~\ref{cor got by Riemann-Roch  theorem}, one can get
\begin{lem}\label{mor under automorphism}
 Let $F\in\Aut(\md^b(\X))$, $X,Y\in\ind\X$  such that $\mu X<\mu Y$. Suppose that $F(X)\in\coh\X[l], F[Y]\in\coh\X[m]$ for some integers $l,m$. Then either $m=l$ or $m=l+1$. Moreover, if $m=l$, then $\mu(F(X))<\mu(F(Y))$. Else if  $m=l+1$, then $\mu(F(X))>\mu(F(Y))$.
\end{lem}
\begin{proof}
Firstly, we claim that $F(\mc_{\mu X})=\mc_{\mu(F(X))}[l]$.

Let $X'$ be an indecomposable object lying in $\mc_{\mu X}$. Suppose that $F(X')\in\coh\X[l']$ for some integer $l'$. Since  $\mu X= \mu X'$, by Lemma~\ref{LM2-1}, we have $\mu(F(X))= \mu(F(X'))$.

Let $Z$ be an indecomposable object lying  in a homogenous tube in $\coh\X$ such that $\mu Z<\mu X$. By Lemma~\ref{LM2-1}, we have $\mu(F(X))\neq \mu(F(Z))$. By Lemma~\ref{mor from homogenous tube}, we have $\Hom_{\X}(Z,X)\neq 0$ and $\Hom_{\X}(Z,X')\neq 0$. Then $\Hom_{\md^b(\X)}(F(Z),F(X))\neq 0$ and $\Hom_{\md^b(\X)}(F(Z),F(X'))\neq 0$.  Since $\coh\X$ is hereditary and $\mu(F(X'))= \mu(F(X))\neq \mu(F(Z))$, we must have $l'=l$. Then  $F(X')\in\mc_{\mu(F(X))}[l]$. Hence,  $F(\mc_{\mu X})\subset\mc_{\mu(F(X))}[l]$. Let $G$ be the automorphism of $\md^b(\X)$ such that $G\circ F=Id$. Since any automorphism of $\md^b(\X)$ is commuted with the translation functor $[1]$, by $F(X)[-l]\in \coh\X$, $G (F(X)[-l])=X[-l]\in \coh\X[-l]$, one can get $$G(\mc_{\mu(F(X))}[l])=G(\mc_{\mu(F(X)[-l])})[l]\subset \mc_{\mu(G\circ F(X))}[-l][l]=\mc_{\mu X}.$$ Therefore, we must have $F(\mc_{\mu X})=\mc_{\mu(F(X))}[l]$ since $G\circ F=Id$. Similarly, we also have $F(\mc_{\mu Y})=\mc_{\mu(F(Y))}[m]$.

In the next, we are going to prove the Lemma.

 Since $\mu X< \mu Y$, by Lemma~\ref{LM2-1}, we have $\mu(F(X))\neq \mu(F(Y))$ while  by Lemma~\ref{cor got by Riemann-Roch  theorem}, we have $\Hom_{\X}(\mc_{\mu X},\mc_{\mu Y})\neq 0$. Then
$$\Hom_{\md^b(\X)}(\mc_{\mu(F(X))}[l],\mc_{\mu(F(Y))}[m])\cong\Hom_{\md^b(\X)}(F(\mc_{\mu X}),F(\mc_{\mu Y}))\cong \Hom_{\X}(\mc_{\mu X},\mc_{\mu Y}) \neq 0,$$ then $m=l$ or $m=l+1$.

If $m=l$, then 
\begin{eqnarray*}
0\neq\Hom_{\md^b(\X)}(\mc_{\mu(F(X))}[l],\mc_{\mu(F(Y))}[m])&\cong&\Hom_{\md^b(\X)}(\mc_{\mu(F(X))},\mc_{\mu(F(Y))})\\
&\cong& \Hom_{\X}(\mc_{\mu(F(X))},\mc_{\mu(F(Y))}).
\end{eqnarray*}
Since $\mu(F(X))\neq \mu(F(Y))$, by Lemma~\ref{cor got by Riemann-Roch  theorem}, $\mu(F(X))<\mu(F(Y))$.

If $m=l+1$, then 
\begin{eqnarray*}
0\neq\Hom_{\md^b(\X)}(\mc_{\mu(F(X))}[l],\mc_{\mu(F(Y))}[m])&\cong&\Hom_{\md^b(\X)}(\mc_{\mu(F(X))},\mc_{\mu(F(Y))}[1])\\
&\cong& \D\Hom_{\X}(\mc_{\mu(F(Y))},\tau \mc_{\mu(F(X))})
\end{eqnarray*}
Then by $\mc_{\mu(F(X)}$ is stable under Auslander-Reiten translation $\tau$, we have 
\begin{eqnarray*}
0\neq\Hom_{\md^b(\X)}(\mc_{\mu(F(X))}[l],\mc_{\mu(F(Y))}[m])&\cong& \D\Hom_{\X}(\mc_{\mu(F(Y))},\tau \mc_{\mu(F(X))})\\
&\cong& \D\Hom_{\X}(\mc_{\mu(F(Y))},\mc_{\mu(F(X))}).
\end{eqnarray*}
Since $\mu(F(X))\neq \mu(F(Y))$, by Lemma~\ref{cor got by Riemann-Roch  theorem}, $\mu(F(Y))<\mu(F(X))$.
\end{proof}


{The following lemma is useful to construct matrices in $\SL(2,\Z)$.}

\begin{lem}\label{exist of abcd}
For integers $a,b$ such that $0< a<b$ and $(a,b)=1$, there are integers $c,d$ such that $bc-ad=1$ and $0<c\leq d<b$ and $c\leq a$.
\end{lem}
\begin{proof} 
Since $(a,b)=1$, so there are integers $c_0,d_0$ such that $bc_0-ad_0=1$.  Then  $b(c_0+ma)-a(d_0+mb)=1$ for any integer $m$. Choose an integer $m$ such that $0<d_0+mb\leq b$, let $c=c_0+ma,d=d_0+mb$. If $d=b$, then $1=bc-ad=b(c-a)\neq 1$ since $b\ge 2$, a contradiction. So $0<d< b$. By $d>0$ and $bc-ad=1$, we have $c>0$. If $c>a$ or $c>d$, then $bc-ad>1$. Hence, $c\leq d$ and $c\leq a$.
\end{proof}


\subsection{Perpendicular category}
Let $\ma$ be an abelian category and $\ms$ be a system of objects in $\ma$. The right perpendicular category $\ms^{\perp}$ of $\ms$ is defined as the full subcategory of all objects $M\in\ma$ which satisfy both
$\Hom_{\ma}(S,M)=0$ and $\Ext_{\ma}^1(S,M)=0$
for all $S\in\ms$. {Similarly,} one can define the left perpendicular category $^{\perp}\ms$ of $\ms$.

\begin{lem}\cite{Hu} \label{perpendicular category}
Let $\X$ be a weighted projective line and $X\in\coh\X$ be an indecomposable rigid vector bundle, then $^\perp X$ and $X^{\perp}$ are module category of some finite dimensional hereditary algebra {with $n-1$ simple objects} where $n$ is the rank of $K_0(\X)$.
\end{lem}

\begin{lem}\cite{GL2}\label{perpendicular of simple object}
Let $\X$ be a weighted projective line of tubular type with weight sequence $(p_1,p_2,\cdots,p_t)$ and $S_i\in\coh\X$ be an exceptional  simple sheaf lying in a tube with rank $p_i$, then $^\perp S_i$ and $S_i^{\perp}$ are the hereditary category of coherent sheaves of  domestic type  of $(p_1,\cdots, p_i-1, \cdots,p_t)$.
\end{lem}

Let $\X$ be a weighted projective line. Let $Z$ be an indecomposable object in $\md^b(\X)$, denote by  $Z^{\perp_\md}$  the full subcategory of all objects  $M\in \md^b(\X)$ such that $\Hom_{\md^b(\X)}(Z[i],M)=0$ for any $i\in \Z$. Since $\coh\X$ is hereditary, it is easy to get that $Z^{\perp_\md}$ in $\md^b(\X)$ is the derived category of $Z^{\perp}$ in $\coh\X$. Similarly,  we can define $^{\perp_\md}Z$. 

\begin{prop}\label{perpendicular cat is tame hereditary}
Let $\X$ be a tubular weighted projective line and $X\in\coh\X$  a quasi-simple rigid vector bundle, then  $X^{\perp}$  is the module category of some connected  hereditary algebra of tame type.
{Similar result holds true for $^\perp X$.}
\end{prop}
\begin{proof}
Let  $S\in\coh\X$  be  an exceptional simple sheaf such that $S$ has same $\tau$-period with $X$.  By Theorem~\ref{telescopic functor} and Lemma~\ref{slop-preserving functor}, we may choose an automorphism $F$ of $\md^b(\X)$ such that $F(S)=X$. Consequently, $S^{\perp_{\md}}$ is equivalent to  $X^{\perp_{\md}}$. Note that  $S^{\perp_{\md}}$  is the derived category of $S^{\perp}$ and $X^{\perp_{\md}}$  is the derived category of $X^{\perp}$.
By Lemma~\ref{perpendicular category} and  Lemma~\ref{perpendicular of simple object},  $X^{\perp}$  is the module category of some   hereditary algebra $H$, $S^{\perp}$ is the category of coherent sheaves of  domestic type.  Since the  category of coherent sheaves of  domestic type is derived equivalent to the module category of some connected hereditary algebra $A$ of tame type,  the two hereditary algebras $H$ and $A$ are derived equivalent. Hence, $H$  is also a connected hereditary algebra of tame type.
\end{proof}

For a tame hereditary algebra $A$, it is well-known that there is no nonzero map from the  preinjective component  to the regular and  the preprojective  components and   there is no nonzero map from the regular component  to the  preprojective  component. And it is known that every preprojective $A$-module has a nonzero map to each homogenous  tube and every homogenous tube has a nonzero map to each preinjective $A$-module, then using  Lemma~\ref{mor from homogenous tube} and Proposition~\ref{perpendicular cat is tame hereditary}, one can get:
\begin{lem}\label{component}
Let $\X$ be a tubular weighted projective line, $X$ be a  quasi-simple rigid vector bundle  in $\coh\X$. Let $Z$ be an indecomposable object in $\coh\X$ such that $Z\in X^{\perp}$ (resp. $^{\perp}X$).
 \begin{itemize}

\item[(1)] If $\mu Z<\mu X$, then $Z$ lies in the preprojective  component of $X^{\perp}$ (resp. $^{\perp}X$).
 \item[(2)] If $\mu Z=\mu X$, then $Z$ lies in the regular component of $X^{\perp}$ (resp. $^{\perp}X$).
\item[(3)]   If $\mu Z>\mu X$, then $Z$ lies in the preinjective  component of $X^{\perp}$ (resp. $^{\perp}X$).
    \end{itemize}
\end{lem}

\begin{proof}
 It is clear that the homogenous tubes in the subcategory $\mc_{\mu X}$ belong to $X^{\perp}$ and they are still homogenous  in $X^{\perp}$. Choose  one  in these homogenous tubes, denoted by $\mt_{thin}$.

 If $\mu Z<\mu X$,  by Lemma~\ref{mor from homogenous tube}, $\Hom_{\X}(Z,\mt_{thin})\neq 0$. Then $\Hom_{X^{\perp}}(Z,\mt_{thin})\neq 0$, hence $Z$ must lie in the preprojective  component of $X^{\perp}$.

  If $\mu Z>\mu X$,  by Lemma~\ref{mor from homogenous tube},  $\Hom_{\X}(\mt_{thin},Z)\neq 0$. Then $\Hom_{X^{\perp}}(\mt_{thin},Z)\neq 0$, hence $Z$ must lie in the preinjective  component of $X^{\perp}$.

  Now suppose $\mu Z=\mu X$. If $Z\in \mt_{thin}$, then obviously $Z$ lies in the regular component of $X^{\perp}$.
Else $\Hom_{\X}(\mt_{thin},Z)= 0=\Hom_{\X}(\mt_{thin},Z)$, then $\Hom_{X^{\perp}}(\mt_{thin},Z)=0=\Hom_{X^{\perp}}(Z,\mt_{thin})$. Since the tame hereditary algebra whose module category is equivalent to $X^{\perp}$ is   connected,  $Z$ lies in the regular component of $X^{\perp}$.
 \end{proof}


For a tame hereditary algebra $A$,  a {\it preprojective tilting} $A$-module  is a tilting $A$-module  lying  in the preprojective component of $\mod A$. {Similarly, }we can define a {\it preinjective tilting} $A$-module.

\begin{lem}\label{tilting object from perpendicular category to hereditary category}
Let $\X$ be a tubular weighted projective line and $X\in\coh\X$ a quasi-simple rigid vector bundle. 
\begin{itemize}
\item[(1)] Let $M$ be a preprojective tilting module in $X^{\perp}$, then $X\oplus M$ is a tilting object in $\coh\X$. 
\item[(2)] Let $N$ be a preinjective tilting module in $^{\perp}X$, then $X\oplus N$ is a tilting object in $\coh\X$.
\end{itemize}
\end{lem}
\begin{proof} {We prove $(1)$ and the proof of $(2)$ is similar.}

Since $M\in X^{\perp}$,  $\Ext^1_{\X}(X,M)=0$. Note that, as $M$ lies in the preprojective  component of $X^{\perp}$, each direct summand of $M$ has slope smaller than $\mu X$ by Lemma~\ref{component}. As a consequence, $\Ext_{\X}^1(M,X)=0$. Therefore, we have $\Ext^1_{\X}(X\oplus M,X\oplus M)=0$. Since $|M|=n-1$, hence $|X\oplus M|=n$, then $X\oplus M$ is a tilting object in $\coh\X$. \end{proof}



\section{Bundle-mutation}\label{bundle-mutation}

\subsection{$\APR$ mutation and co-$\APR$ mutation}
Let $\mh$ be a hereditary abelian category with tilting objects.
\begin{df}
 Let  $T=\bigoplus_{i=1}^n T_i$ be a tilting object in $\mh$ where $T_i\in\ind \mh$ for each $i$, $T_k$ and $T_l$ be two indecomposable direct summands of $T$.
\begin{itemize}
\item[(1)]If $\Hom_{\X}(T_i,T_k)=0$ for any $i\neq k$, we call $T_k$ a first  object of $T$.
\item[(2)]If $\Hom_{\X}(T_l,T_i)=0$ for any $i\neq l$, we call $T_l$ a last object of $T$.
\end{itemize}
\end{df}

\begin{df}
For a tilting object $T$ in $\mh$,  we call the mutation at a first object of $T$ an $\APR$ mutation and the mutation at a last object of $T$ a co-$\APR$ mutation.
\end{df}

The following Lemma is easy to get ({\it cf.}~\cite[Chapter VIII, Theorem 4.5]{ASS} and \cite{P,BKL}).
 \begin{lem}\label{APR mutation on preprojective tilting module}
Let $\mh$ be the module category of  a  tame hereditary algebra, $T$ and $T'$ be two tilting modules in $\mh$.
\begin{itemize}
\item[(a)] Suppose that $T'$ is obtained from $T$ by  an $\APR$ mutation or a co-$\APR$ mutation. If $T$ is preprojective, then $T'$ is also preprojective, while if $T$ is preinjective, then $T'$ is also preinjective.

\item[(b)] If $T$ is  preprojective, then by  $\APR$ mutations, $T$ can be transformed into a slice lying in the preprojective component.
While if $T$ is  preinjective, then by  co-$\APR$ mutations, $T$ can be transformed into a slice lying in the preinjective component.

 \item[(c)]If both $T$ and $T'$ are preprojective tilting modules in $\mh$ or preinjective tilting modules in $\mh$, then $T$ can be get from $T'$ by  $\APR$ mutations and co-$\APR$ mutations.

\end{itemize}
\end{lem}

Then by Lemma~\ref{tilting object from perpendicular category to hereditary category}, we have:
\begin{lem}\label{pp-mutation is mutation of h}
Let $\X$ be a tubular weighted projective line,  $X$ be a quasi-simple rigid vector bundle.  Then
\begin{itemize}
\item[(1)] the $\APR$ mutation  on a preprojective tilting module in $X^{\perp}$ {induces } an $\APR$ mutation  in $\coh\X$ naturally;
\item [(2)]  the co-$\APR$ mutation on a preinjective tilting module in $^{\perp}X$ induces  a  co-$\APR$ mutation in $\coh\X$ naturally.
\end{itemize}
\end{lem}

\subsection{Minimal  and maximal direct summands of a tilting sheaf}
Let $\X$ be a tubular weighted projective line, for some special direct summands of a tilting sheaf, we give the following definitions for convenience.
\begin{df}
 Let $T=\bigoplus_{i=1}^n T_i$ be a tilting sheaf in $\coh\X$ where $T_i\in\ind\X$ for each $i$. Let $T_k$, $T_l$ are two indecomposable direct summands of $T$.
\begin{itemize}
 \item[(1)]If   $\mu T_k\leq \mu T_i$ for any $1\leq i\leq n$,  we say that $T_k$ is a minimal direct summand of $T$.
Moreover, if $\mu T_k<\mu T_i$ for any $i\neq k$, we say that $T_k$ is the only minimal  direct summand of $T$.
 \item[(2)] If  $\mu T_l\ge \mu T_i$ for any $1\leq i\leq n$,   we say that $T_l$ is a maximal direct summand  of $T$.
Moreover, if $\mu T_l>\mu T_i$ for any $i\neq l$,   we say that  $T_l$ is the only maximal  direct summand of $T$.
\item[(3)] Let $T_k$ be a minimal direct summand of $T$, $T_l$ be a maximal direct summand  of $T$, we call the interval $[\mu T_k, \mu T_l]$ the slope range of $T$.
\end{itemize}
\end{df}
Note that the first object of a tilting sheaf may not be a minimal direct summand, and a minimal direct summand of a tilting sheaf also may not be a first object. But the  only minimal  direct summand of $T$  must be a first object.

\begin{lem}\label{tiling module and tilting complex}
Let   $T=X\oplus \overline{T}$ be a tilting sheaf in $\coh\X$ where $X\in\ind\X$.
\begin{itemize}
 \item[(1)] If $X$ is the only maximal  direct summand of $T$, then $T'=\tau X[-1]\oplus \overline{T}$ is a tilting complex in $\md^b(\X)$.
 \item[(2)]  If $X$ is the only minimal  direct summand of $T$, then $\tauni X[1]\oplus \overline{T}$ is a tilting complex in $\md^b(\X)$.
\end{itemize}
\end{lem}
\begin{proof} We only prove $(1)$ here, and $(2)$ could be deduced similarly.

For $(1)$, since $X$ is the only maximal  direct summand of $T$,  $\mu (\tau X)=\mu X>\mu T_i$ for any direct summand $T_i$ of $\overline{T}$, then $$\Hom_{\md^b(\X)}(\tau X[-1], \overline{T}[-1])\cong\Hom_{\X}(\tau X, \overline{T})=0.$$
Then using  $T=X\oplus \overline{T}$ be a tilting sheaf in $\coh\X$, one can get that $\Hom_{\md^b(\X)}(T',T'[l])=0$ for any $l\neq 0$ easily. Since the number of non-isomorphic  indecomposable direct summands of $T'$ is $n$, we conclude that $T'$ is a tilting complex in $\md^b(\X)$ by \cite[Lemma 9.1.3]{M2}.
\end{proof}

\subsection{The behavior of slopes under mutations}
Let $\X$ be a tubular weighted projective line. In this subsection, we investigate how the slope changed when we make mutations at some special direct summands of  a tilting sheaf.

\begin{lem}\label{mutation at quasi-length}
Let $T=T_k\oplus\overline{T}$ be a tilting object in $\coh\X$ with an indecomposable direct summand $T_k$. Let $\mu_{T_k}(T)=T_k^*\oplus\overline{T}$ be the mutation of $T$ at $T_k$.
If  there is  another direct summand of $T$,  denoted by $Z$, such that   $T_k\in\mw_Z$, then $T_k^*\in \mw_Z$.  {Otherwise, $T_k^*$ and $T_k$ have different slopes.}
\end{lem}
\begin{proof}
{Assume that there is a direct summand $Z$ of $T$ such that $T_k\in \mw_Z$. Assume moreover that $q.l. Z=l$.}
 By Lemma~\ref{tilting in wing},  the summands of $T$ lying in $\mw_Z$ can be {viewed} as a tilting module in $\mod A_l$, where $A_l$ is the path algebra of a quiver of Dynkin type $\mathrm{A}_l$ with linear orientation. Note that, as $Z$ is both a projective and injective object in $\mw_Z$,  the tilting modules in $\mw_Z$ always have $Z$ as a direct summand. Since $Z$ is a sincere $A_l$-module, every almost tilting module in $\mw_Z$ containing $Z$ as a direct summand has two complements in $\mw_Z$.   Hence
 if  $T_k\in\mw_Z$ and $Z\ncong T_k$, we can deduce that $T_k^*\in \mw_Z$ by Lemma \ref{tilting in wing}.

 Now assume that there is no such direct summand.
By Lemma~\ref{induce}, there is an exact sequence  $0\ra T_k\xra{f} B\ra T_k'\ra 0$ or  $0\ra T_k^*\ra B'\xra{g} T_k\ra 0$ in $\coh\X$, where $f$ is a minimal left $\add\overline{T}$-approximation and $g$ is a minimal right $\add\overline{T}$-approximation.  Suppose $\mu T_k^*=\mu T_k$. 
Then  each direct summand of $B$ and $B'$ must have the same slope as $T_k$.   If $T_k^*$ and $T_k$ belong to different tubes,  by $\Hom_{\X}(\mt_{T_k},\mt_{T_k^*})=0=\Hom_{\X}(\mt_{T_k^*},\mt_{T_k})$,  neither of the exact sequences $0\ra T_k\xra{f} B\ra T_k'\ra 0$ or  $0\ra T_k^*\ra B'\xra{g} T_k\ra 0$ will hold.
   So $T_k^*$ lies in the same tube as $T_k$,  then  there must be one direct summand of $B$ or $B'$, denoted by $Z$, such that $T_k$ has a monomorphism to $Z$ or $Z$ has an epimorphism to $T_k$. In any case, $T_k\in \mw_Z$, which contradicts with the assumption. Therefore, $T_k$ and $T_k^*$ have different slopes.
   \end{proof}

\begin{thm}\label{slop change under APR and coAPR mutation}
Let $T=\bigoplus_{i=1}^n T_i$ be a tilting sheaf in $\coh\X$ with $\mu T_i \leq \mu T_j$ for $1\leq i<j\leq n$. For $1\leq k\leq n$, let  $\mu_{T_k}(T)=T_k^*\oplus (\bigoplus_{i\neq k} T_i)$ be the mutation of $T$ at $T_k$. Then
\begin{itemize}
 \item[(1)] $\mu T_1\leq \mu T_1^*\leq \mu T_n$ and $\mu T_1\leq \mu T_n^*\leq \mu T_n$.


\item[(2)] If $T_k$ is a first object of $T$, then $\mu T_k\leq \mu T_k^*\leq \mu T_n$.
\item[(3)] If $T_k$ is a last object of $T$, then $\mu T_k\ge \mu T_k^*\ge\mu T_1$.

\end{itemize}
\end{thm}
\begin{proof}  For $(1)$,
if $\mu T_1^*< \mu T_1$, by Lemma~\ref{induce}, there is an exact sequence $$0\ra T_1^*\ra B\xra{g} T_1\ra 0$$  in $\coh\X$, where $g$ is a minimal right $\add(\bigoplus_{i=2}^n T_i)$-approximation. Since $\mu T_1 \leq \mu T_j$ for $1\leq j\leq n$,  all direct summands of $B$ must belong to the same tube with $T_1$. Then there is a direct summand $Z$ of $B$ such that $Z$ has an epimorphism to $T_1$, then $T_1\in\mw_Z$. By Lemma~\ref{mutation at quasi-length}, $T_1^*\in\mw_Z$. This contradicts the assumption that  $\mu T_1^*< \mu T_1$.
Hence $\mu T_1^*\ge \mu T_1$.

 If $\mu T_1^*> \mu T_n$, then $T_1^*$ is the only  maximal direct summand of  the tilting sheaf $T_1^*\oplus (\bigoplus_{i=2}^n T_i)$. By Lemma~\ref{tiling module and tilting complex},  we know $\tau T_1^*[-1]\oplus (\bigoplus_{i=2}^n T_i)$ is a tilting complex in $\md^b(\X)$. By Theorem~\ref{telescopic functor}, we can choose an automorphism $F$ of $\md^b(\X)$  such that   $F(T_n)\in\mc_{\infty}$.
 Hence, $F(T)$ and $F(\tau T_1^*[-1]\oplus (\bigoplus_{i=2}^n T_i))$ are two tilting objects in $\md^b(\X)$.
 Since for any $1\leq i<j\leq n$, $\mu T_i\leq\mu T_j\leq \mu T_n<\mu T_1^*=\mu(\tau T_1^*)$, by Lemma~\ref{mor under automorphism}, we have $ F(\tau T_1^*)\in\coh\X[1]$, \ie $ F(\tau T_1^*[-1])\in\coh\X$ and $F(T_i)\in\coh\X$ and  $$\mu(F(\tau T_1^*[-1]))<\mu (F(T_1))\leq \mu (F(T_i)) \leq \mu (F(T_j))\leq \mu (F(T_n)).$$    Now, we have $F(\tau T_1^*[-1]\oplus (\bigoplus_{i=2}^n T_i))$ and $F(T)$ are two  tilting objects in $\coh\X$. Moreover, by $F(T_1)$ and $F(\tau T_1^*[-1])$ are two complements of the almost tilting object $F(\bigoplus_{i=2}^n T_i)$. By $  \mu (F(T_i)) \leq \mu (F(T_j))$ for any $1\leq i<j\leq n$,  we have $\mu (F(\tau T_1^*[-1]))\ge\mu (F(T_1))$ in $\coh\X$, 
  which is in a contradiction with $\mu(F(\tau T_1^*[-1]))<\mu (F(T_1))$. Therefore, we must have $ \mu T_1^*\leq \mu T_n$.
Using similar proof, we can get the results for $\mu T_n^*$.

For $(2)$,  by Lemma~\ref{induce}, there is an exact sequence  $0\ra T_k\xra{f} B\ra T_k^*\ra 0$ or  $0\ra T_k^*\ra B'\xra{g} T_k\ra 0$ in $\coh\X$, where $f$ is a minimal left $\add(\bigoplus_{i\neq k} T_i)$-approximation and $g$ is a minimal right  $\add(\bigoplus_{i\neq k} T_i)$-approximation.
Since $T_k$  is the first object of $T$, then only the exact sequnece  $0\ra T_k\ra B\xra{g} T_k^*\ra 0$  can happen.  So $\mu T_k\leq \mu T_k^*$.

If $ \mu T_k^*> \mu T_n$,  similar to the case $(1)$, one can get a contradiction. Hence, $ \mu T_k^*\leq \mu T_n$,

 {The proof of $(3)$ is similar to $(2)$ and we omit it here.}
\end{proof}

\subsection{Bundle-mutations}
Let $\X$ be a tubular weighted projective line.   Let  $T$ and $T'$ be two tilting objects in $\coh\X$ such that    $T'$ is a mutation of $T$,  if $T,T'$ are both tilting bundles,  we call such mutation a {\it bundle-mutation}.

Let $T$ and $T'$ be tilting bundles in $\coh\X$, we say that $T$ can be transformed into $T'$ by bundle-mutations if there is a sequence of tilting bundles $M_0=T,M_1,\cdots,M_{l-1},M_l=T'$ such that $M_{i}$ is a mutation of $M_{i-1}$ for any $1\leq i\leq l$.
Note that if $T$ can be transformed into $T'$ by bundle-mutations, then $T'$ also can be transformed into $T$ by bundle-mutations.  

 As a direct consequence of Theorem~\ref{slop change under APR and coAPR mutation}, we have
\begin{lem}\label{APR and co-APR mutations are l mutations}
 Both the $\APR$ mutation and co-$\APR$ mutation on a tilting bundle in $\coh\X$  are bundle-mutations.
\end{lem}

Then combining  Lemma~\ref{component}, Lemma~\ref{tilting object from perpendicular category to hereditary category} and Lemma~\ref{pp-mutation is mutation of h}, one can get
\begin{lem}\label{APR and co-APR mutations are bundle mutations}
Let  $X$ be quasi-simple rigid vector bundle in $\coh\X$.  Then
 the $\APR$ mutation  and co-$\APR$ mutation  in $X^{\perp}$  on a preprojective tilting module induce bundle-mutations in $\coh\X$, while the co-$\APR$ mutation  in $^{\perp}X$ on a preinjective tilting module   which is a vector bundle in $\coh\X$ is also a bundle-mutation in $\coh\X$.
\end{lem}

\section{Connectedness of tilting bundles}\label{connectedness of tilting bundles}
Let $\X$ be a tubular weighted projective line with weight sequence $(p_1,\cdots,p_t)$, up to permutation,  we   assume that $p_i\leq p_j$ for $i<j$. Set $p:=\lcm(p_1,\cdots,p_t)$.  Recall that the tubular weight types are, $(2,2,2,2),(3,3,3),(2,4,4)$ and $(2,3,6)$. It follows that $p=p_t$, $\delta(\vec{x}_t)=p/p_t=1$ and  $\mu (\co(m\vec{x}_t))=\delta(m\vec{x}_t)=m\delta(\vec{x}_t)=m$ for any integer $m$. Denote by $T_{can}=\bigoplus_{0\leq \vec{x}\leq \vec{c}}\co(\vec{x})$ the canonical tilting bundle in $\coh\X$.
 \subsection{Some tilting bundles}
\begin{lem}\label{tilting bundle passing X1}
Let $X$  be a rigid vector bundle in $\coh\X$,  then there is a tilting bundle with $X$ as a direct summand.
\end{lem} 
\begin{proof} 
Since $X$ is a rigid object in $\coh\X$, we can choose a tilting object $T$ in $\coh\X$ with $X$ as a direct summand. If $T$ is a tilting bundle, we are done. Else suppose that $X_1,\cdots,X_s (s\ge 1)$ are all the indecomposable direct summands of $T$ which lie in $\mc_{\infty}$ and  $q.l. X_i\leq q.l. X_{i+1}$ for $1\leq i\leq s-1$.
Let $T'=\mu_{X_1}\cdots\mu_{X_s}(T)$.
By Lemma~\ref{mutation at quasi-length}, $T'$ has no direct summand lying in $\mc_{\infty}$, so $T'$ is a tilting bundle.  Since $X$ is a vector bundle,  $X\neq X_i$ for any $1\leq i\leq s$,  $X$ is still a direct summand of $T'$. Hence, $T'$ is the tilting bundle we needed.
 \end{proof}
 
\begin{lem}\label{tilting bundle passing X}
Let $X$ be a quasi-simple rigid vector bundle in $\coh\X$, then
\begin{itemize}
\item[(a)] there is a tilting bundle  {such that}   $X$ is the only minimal  direct summand;
 \item[(b)]  there is a tilting bundle {such that}  $X$ is  the only maximal  direct summand;
  \item[(c)] if the $\tau$-period of $X$ is $p$, then  there is a tilting bundle $T$ such that $X$ is a direct summand of $T$ and  $\mu X$ is neither minimal nor maximal among the slopes of all the direct summands of $T$.
  \end{itemize}
\end{lem}
\begin{proof}
For $(a)$, note that $^{\perp}X$  is the module category of a  tame hereditary algebra.  By Lemma~\ref{component},  it is easy to get that there are only finite indecomposable objects lying in $^{\perp}X \cap \mc_{\infty}$.  Hence, we can choose a preinjective tilting module  $M_1$ in $^{\perp}X$ such that $ M_1$ is a vector bundle in $\coh\X$. Then  by Lemma~\ref{component}  and Lemma~\ref{tilting object from perpendicular category to hereditary category},  $X\oplus M_1$ is a tilting bundle in $\coh\X$ such that $X$ is  the only minimal direct summand of this tilting bundle.

 For $(b)$,  consider $X^{\perp}$, then the proof is similar to $(a)$.

For $(c)$, we will discuss it in the following two cases.

 Case 1: $X$ is a line bundle $\co(\vec{x})$ for some $\vec{x}\in L(\bf p)$. Let $T=T_{can}(\vec{x}-\vec{x}_t)$,  then $T$ is a tilting bundle, $\co(\vec{x}-\vec{x}_t),\co(\vec{x}+\vec{x}_t)$ are direct summands of $T$ and   $\mu (\co(\vec{x}-\vec{x}_t))<\mu( \co(\vec{x}))<\mu (\co(\vec{c}+\vec{x}+\vec{x}_t))$, hence $T_{can}(\vec{x}-\vec{x}_t)$ is the tilting bundle we needed.

Case 2: $X$ is not a line bundle. Since each line bundle in $\coh\X$ has integer slope, we can suppose that $\mu X= m+\frac{a}{b}$ for some  integers $m,a,b$ such that $(a,b)=1$ and $0<a<b$. Then there are positive integers $c,d$ such that $bc-ad=1,0<c\leq d< b, c\leq a$. Let $\varphi(x)=\frac{(c+dm)x+(a+bm)}{dx+b}$.  By Lemma~\ref{LM2},  there is an automorphism of $\md^b({\X})$, denoted by $\overline{\varphi}$,  such that $\overline{\varphi}$ coincides with the action of   $\varphi$ on slopes. Since  $X$ is  a quasi-simple  object in $\coh\X$ with $\tau$-period  $p$,  we may assume that $\overline{\varphi}(\co)=X$ by Lemma~\ref{slop-preserving functor}.  Since $T_{can}(-\vec{x}_t)$ is a tilting object in $\coh\X$, $\overline{\varphi}(T_{can}(-\vec{x}_t))$ is a tilting object in $\md^b(\X)$.
By $$\varphi(\mu (\co(-\vec{x}_t)))=\varphi(-1)=m+\frac{a-c}{b-d} < m+\frac{c}{d}=\varphi(\infty)$$ 
and  $\co(-\vec{x}_t)$ is a minimal direct summand  of $T_{can}(-\vec{x}_t)$,
by Lemma~\ref{mor under automorphism}, all direct summands of $\overline{\varphi}(T_{can}(-\vec{x}_t))$ have slopes smaller than $\varphi(\infty)=m+\frac{c}{d}$,
hence  $\overline{\varphi}(T_{can}(-\vec{x}_t))$ is a tilting object in $\coh\X$, moreover, a tilting bundle in $\coh\X$. As $\co$ is a direct summand of $T_{can}(-\vec{x}_t)$,  $X=\overline{\varphi}(\co)$ is a direct summand of  $\overline{\varphi}(T_{can}(-\vec{x}_t))$. Hence, the tilting bundle $\overline{\varphi}(T_{can}(-\vec{x}_t))$ is what we wanted. 
\end{proof}

\subsection{Connectedness in some special cases}

\begin{lem}\label{case common object which is the last object}
Let $T,T'$ be two tilting bundles in $\coh\X$ such that $T$ and $T'$ have a  common quasi-simple  direct summand $X$.
  Then $T$ can be transformed into $T'$ by bundle-mutations if $X$ {satisfies} one of the following two conditions:
\begin{itemize}
\item[(1)]$X$ is the only minimal  direct summand of $T$ and $T'$.
\item[(2)] $X$ is  the only maximal  direct summand of $T$ and $T'$.
\end{itemize}
\end{lem}
\begin{proof}
We give a proof for the condition $(1)$ and the proof for the condition $(2)$ is similar.

 Denote by $T=X\oplus \overline{T}$ and $T'=X\oplus \overline{T'}$. According to Lemma~\ref{component},  both $\overline{T}$ and $\overline{T'}$   lie on the preinjective  component of $^{\perp} X$.  According to Lemma~\ref{APR mutation on preprojective tilting module}, by co-$\APR$ mutations, $T$ can be transformed into a slice $\Sigma$, $T'$ also can be transformed into a slice $\Sigma'$.
Choose a slice $\Sigma''$ on the preinjective  component such that $\Fac \Sigma'\subset \Fac \Sigma''$ and $\Fac \Sigma\subset \Fac \Sigma''$. Then  by  co-$\APR$ mutations, $\Sigma$  and $\Sigma'$ can be transformed into  $\Sigma''$.
By Lemma~\ref{APR and co-APR mutations are bundle mutations},  all the mutation appeared above are bundle-mutations, hence, $T$ and $T'$ can be transformed into each other by bundle-mutations.
\end{proof}

\begin{lem}\label{case two last object in different tube}
Let $T$ and $T'$ be two tilting bundles  in $\coh\X$. Let $X$ and $ X'$ be two  quasi-simple  objects  with same slope but belong to different tubes in $\coh\X$. If $X$ and $X'$ satisfy  one of the following two conditions:
\begin{itemize}
\item[(1)]   $X$ is the only minimal  direct summand of $T$ and  $X'$ is the only minimal  direct summand of $T'$,
\item[(2)]$X$ is the only maximal  direct summand of $T$ and $X'$ is the only maximal  direct summand of $T'$;
\end{itemize}
 then $T$ can be transformed into $T'$ by bundle-mutations.
\end{lem}
\begin{proof} We give a proof of for the condition $(1)$ and omit the proof for the condition $(2)$.

Denote by $\mh'=^{\perp}X$.
Note that, as $\mu X=\mu X'$ and $X, X'$ belong to different tubes, we have $\Ext^1_{\X}(X',X)=\Hom_{\X}(X',X)=0$. Consequently,   $X'\in\mh'$ is a  quasi-simple regular module in $\mh'$ by Lemma~\ref{component}. It is known that
 the left perpendicular category $^{\perp_{\mh'}}X'$ of $X'$ in $\mh'$ is also a module category of a tame hereditary algebra ({\it cf.}~\cite{SS2}).  Since in $\coh\X$, there are only finite indecomposable objects  lying in $\mc_{\infty}\cap ^{\perp}X$,  we can choose a tilting module $T_v$ in the preinjective component of $^{\perp_{\mh'}}X'$  such that $T_v$ is a vector bundle in $\coh\X$. It is easy to get that  $T''=X\oplus X'\oplus T_v$ is a tilting bundle in $\coh\X$ such that  $X$ and $X'$ are minimal direct summands of $T''$.
 Moreover, $T''$ admits no mininal direct summands other than $X$ and $X'$.
   Let $M_1=\mu_{X'}(T'')$, $M_2=\mu_{X}(T'')$. By Theorem~\ref{slop change under APR and coAPR mutation} and Lemma~\ref{mutation at quasi-length}, $M_1$ and $M_2$ are also tilting bundles,
and  $X$ is the only minimal  direct summand of $M_1$, $X'$ is the only minimal  direct summand of $M_2$.   By Lemma~\ref{case common object which is the last object},
 $T$ can be obtained from $M_1$ by bundle-mutations,  $T'$ can be obtained from $M_2$ by bundle-mutations. Hence $T$ can be obtained from $T'$ by bundle-mutations.
\end{proof}

\begin{lem}\label{last object to first object}
Let $T$ be a  tilting bundle in $\coh\X$ with a quasi-simple direct summand $X$, then
 \begin{itemize}
\item[(1)]  $T$ can be transformed into a  tilting bundle  such that  $X$ is the only minimal  direct summand of  the new tilting bundle by  bundle-mutations.
\item[(2)] $T$ can be transformed into a  tilting bundle   such that  $X$ is the only maximal  direct summand of the new tilting bundle by bundle-mutations.
 \end{itemize}
\end{lem}
\begin{proof}
Let $T_{min}$ be a minimal direct summand of $T$, $T_{max}$ be a maximal direct summand of $T$. Note that $\mu (T_{min})\neq\mu (T_{max})$.

Firstly, we claim that if  $\mu X<\mu (T_{max})$, then $(1)$ holds and  if $\mu X>\mu (T_{min})$, then $(2)$ holds.

Now suppose that  $\mu X<\mu (T_{max})$. Consider the right perpendicular category $T_{max}^{\perp}$ of  $T_{max}$ in $\coh\X$. Denote by $\mh'=T_{max}^{\perp}$. $\mh'$ is the module category of a tame hereditary algebra.
Since  $\mu X<\mu (T_{max})$,  by Lemma~\ref{component}, $X$ lies in the preprojective component of $\mh'$.  Then it is easy to  get that the right perpendicular category  $X^{\perp_{\mh'}}$ of $X$ in $\mh'$ is a hereditary  category of finite type. So the set $\{N\in\ind\X|\mu N< \mu X,\Ext_{\X}^1(X,N)=0=\Ext^1_{\X}(T_{max},N)\}\subset X^{\perp_{\mh'}}$ is finite.
 Hence by $\APR$ mutations,
$T$ can be transformed into a new tilting bundle $L$ {such that} $X$ is a minimal direct summand of $L$. If $X$ is the only minimal direct summand, then we are done. Else  suppose that $\{L_i|1\leq i\leq s\}$ are all minimal direct summands of $L$. Assume moreover that $q.l. L_i\leq q.l. L_{i+1}$ for $1\leq i\leq s-1$. Since $X$ is quasi-simple, we may assume that $L_1=X$. Then let $T'=\mu_{L_{2}}\cdots\mu_{L_s}(L)$, by Theorem~\ref{slop change under APR and coAPR mutation} and Lemma~\ref{mutation at quasi-length},  $X$ is the only minimal  direct summand of $T'$. By Lemma~\ref{APR and co-APR mutations are bundle mutations} and Theorem~\ref{slop change under APR and coAPR mutation},  all the mutations appeared above are bundle-mutations. Hence, $T$ can be transformed into a  tilting bundle $T'$ such that $X$ is the only minimal  direct summand of $T'$ by  bundle-mutations.

For the case $\mu X>\mu (T_{min})$, we may apply a similar proof and we omit it here.

Now we separate the remaining proof into three cases.

Case 1: $\mu (T_{min})<\mu X<\mu (T_{max})$. It follows from the claim directly.


Case 2:   $\mu X=\mu (T_{min})$. Note that, as $\mu (T_{min})\neq \mu (T_{max})$, we have $\mu X<\mu (T_{max})$. It follows from the above claim that $T$ can be  transformed into a tilting bundle $M$ such that $X$ is the only minimal  direct summand of $M$ by bundle-mutations. In order to prove $(2)$, it suffices to prove that $M$ can be transformed into a  tilting bundle such that $X$ is the only maximal  direct summand of the new tilting bundle by bundle-mutations.
 We will discuss it in the following two cases.  

Subcase 2.1:  the $\tau$-period of $X$ is $p$.  By Lemma~\ref{tilting bundle passing X}, there is a tilting bundle $T''$ such that $X$ is a direct summand of $T''$ with $\mu X$ is neither minimal  nor maximal among the slopes of all the direct summands of  $T''$.
By Case 1,  $T''$ can be transformed into a  tilting bundle $M'$ such that  $X$ is the only minimal  direct summand of $M'$ by bundle-mutations and also can be transformed into a tilting bundle $M''$ such that $X$ is the only maximal  direct summand of $M''$ by bundle-mutations.  Now $X$ is the only minimal direct summand of $M$ and $M'$,  by Lemma~\ref{case common object which is the last object}, $M$ can be transformed into  $M'$ by bundle-mutations. In sum, by bundle-mutations, $M$ can be transformed into $M''$ such that  $X$ is the only maximal  direct summand of $M''$.

Subcase 2.2:  the $\tau$-period of $X$ is not $p$.  Choose a quasi-simple rigid object $X'\in\mc_{\mu X}$ such that the  $\tau$-period of $X'$ is $p$. Then we have $X$ and $X'$ belong to different tubes. Let $L$ be  a tilting bundle such that $X'$ is the only minimal  direct summand of $L$.  By Lemma~\ref{case two last object in different tube}, we know that $L$ can be obtained from $M$ by bundle-mutations.
By Subcase 2.1, $L$ can be transformed into a tilting bundle $L'$ such that $X'$ is the only maximal  direct summand of $L'$ by bundle-mutations. Choose a tilting bundle $T'$ such that $X$  is the only maximal  direct summand of $T'$, then again by  Lemma~\ref{case two last object in different tube}, we know that $T'$ can be obtained from $L'$ by bundle-mutations.
In sum, by bundle-mutations, $M$ can be transformed into $T'$ such that  $X$ is the only maximal  direct summand of $T'$.

Case 3:  $\mu X=\mu (T_{max})$. The proof is  similar to the Case 2 and we omit it here.
  \end{proof}

The following result is a direct consequence of Lemma~\ref{last object to first object} and Lemma~\ref{case common object which is the last object}.

\begin{lem}\label{connected in case with a  common direct summand}
Let $T$ and $T'$ be two tilting bundles in $\coh\X$ such that $T$ and $T'$ contain  a common quasi-simple direct summands $X$, then $T$ can be obtained from  $T'$  by bundle-mutations.
\end{lem}

\begin{lem}\label{connected in case with a direct summand in perp}
Let $T$ and $T'$ be two tilting bundles in $\coh\X$, $X$ be a quasi-simple direct summand of $T$ and  $X'$ be a quasi-simple direct summand of $T'$. If $X\oplus X'$ is rigid, then $T$ can be obtained from  $T'$  by bundle-mutations.
\end{lem}
\begin{proof}
Since $X\oplus X'$ is rigid, by Lemma~\ref{tilting bundle passing X1},
we can choose a tilting bundle $T''$ containing $X\oplus X'$ as direct summands. Then $X$ is a common direct summand of $T$ and $T''$, while $X'$ is a common direct summand of $T'$ and $T''$.
By Lemma~\ref{connected in case with a common direct summand}, $T$ can be transformed into  $T'$ by bundle-mutations.
\end{proof}

\subsection{Connectedness  in general case} In this subsection, we will prove that any tilting bundle can be obtained from the canonical tilting bundle $T_{can}$ by bundle-mutations. 
\begin{lem}\label{case of all direct summands are linear bundles}
 Let $T=T_{can}(\vec{x})$ be a tilting bundle in $\coh\X$, then $T$ can be obtained from $T_{can}$ by bundle-mutations.
\end{lem}
\begin{proof}

If $\vec{x}=\vec{x}_i$ for some $1\leq i\leq t$, then $\co(\vec{x}_i)$ is a common quasi-simple direct summand of $T_{can}$ and $T_{can}(\vec{x}_i)$. By Lemma~\ref{connected in case with a common direct summand},   $T$ can be obtained from $T_{can}$ by bundle-mutations.
For the general case, one obtains the result by induction on the degree of $\vec{x}$.

\end{proof}

\begin{lem}\label{case one direct summand is a line bundle}
 Let $T$ be a tilting bundle in $\coh\X$ such that one direct summand of $T$ is a line bundle, then $T$ can be obtained from $T_{can}$ by bundle-mutations.
\end{lem}
\begin{proof}
Without loss of generality, we may assume that $\co(\vec{x})$ is a direct summand of $T$.
 Then $\co(\vec{x})$ is a common quasi-simple direct summand of $T$ and $T_{can}(\vec{x})$.
 By Lemma~\ref{connected in case with a common direct summand} and Lemma~\ref{case of all direct summands are linear bundles},  $T$ can be obtained from $T_{can}$ by bundle-mutations.
\end{proof}

\begin{lem}\label{case the slop range contain an integer}
 Let $T$ be a tilting bundle  in $\coh\X$ such that the slope range of $T$ contained an integer, then $T$ can be obtained from $T_{can}$ by bundle-mutations.
\end{lem}
\begin{proof}
Denote this integer by $m$. 
Let $L$ be the line bundle $\co(m\vec{x}_t)$, then $\mu L=\mu (\co(m\vec{x}_t))=m$.
  By Lemma~\ref{line bundle belong to module canonical algebras}, we have either $\Ext^1_{\X}(T,L)=0$ or $\Hom_{\X}(T,L)=0$. 

Case 1:  $\Ext^1_{\X}(T,L)=0$.   By Lemma~\ref{tilting in wing}, we can choose a quasi-simple direct summand of $T$, denoted by $X$, such that $\mu X\ge m$.
If $X$ is a line bundle, by Lemma~\ref{case one direct summand is a line bundle},  $T$ can be obtained from $T_{can}$ by bundle-mutations.
If  $X$ is not a line bundle, then $X$ and $L$ belong to different tubes since all quasi-simple objects in the tube which $L$ lies in  are line bundles.  Then by $\mu L\leq \mu X$, we must have $\Ext^1_{\X}(L,X)=0$.  Hence,     $X\oplus L$ is rigid. Let  $T'$ be a tilting bundle containing $L$ as a direct summand.  By Lemma~\ref{connected in case with a direct summand in perp},   $T$ can be obtained from $T'$ by bundle-mutations.  We conclude that $T$ can be obtained from $T_{can}$ by bundle-mutations by  Lemma~\ref{case one direct summand is a line bundle}.

Case 2:  $\Hom_{\X}(T,L)=0$. Then $\Ext_{\X}^1(\tauni L,T)\cong \D\Hom_{\X}(T,L)=0$. Notice that $\tauni L$ is  a line bundle and $\mu( \tauni L)=\mu L=m$. Similar to the Case 1, we can obtain the desired result. 
\end{proof}

\begin{lem}\label{extent integer}
Let $X$ be an indecomposable quasi-simple rigid object in $\coh\X$.  Suppose  the  $\tau$-period of $X$ is $p$ and  $\mu X=m+\frac{a}{b}$ for some integers $m,a,b$, where $(a,b)=1, 0<a<b$, then there are positive integers $c,d$ and  a   quasi-simple rigid object $Y$  in $\coh\X$ with  $\tau$-period $p$   such that $bc-ad=1,0<c\leq d< b, c\leq a$,  $\mu Y=m+\frac{c}{d}$ and $X\oplus Y$ is rigid in $\coh\X$.
\end{lem}

\begin{proof}
Since $(a,b)=1, 0<a<b$, by Lemma~\ref{exist of abcd}, there are positive integers $c,d$ such that $bc-ad=1,0< c\leq d< b, c\leq a$. Let $\varphi(x)=\frac{(c+dm)x+(a+bm)}{dx+b}$, since $\left(\begin{array}{cc}c+dm&a+bm \\d&b\end{array}\right)\in\SL(2,\Z)$, by Lemma~\ref{LM2},  there is an automorphism of $\md^b({\X})$, denoted by $\overline{\varphi}$, such that { $\overline{\varphi}$ coincides with the action of ${\varphi}$ on slopes.} Note that $\varphi(0)=m+\frac{a}{b}$. Since  $X$ is  a quasi-simple rigid  object in $\coh\X$ with $\tau$-period $p$  and  $\mu X=m+\frac{a}{b}$,  by Lemma~\ref{slop-preserving functor},  we may assume that $\overline{\varphi}(\co)=X$.  It is clear that there is a simple object $S$ in $\coh\X$ such that the $\tau$-period of $S$ is $p$ and $\co\oplus S$ is rigid. Let $Y=\overline{\varphi}(S)$, then we have $$\mu Y=\mu(\overline{\varphi} (S))=\varphi(\mu S)=\varphi(\infty)=m+\frac{c}{d}>m+\frac{a}{b}=\mu X.$$
By Lemma~\ref{mor under automorphism}, we deduce that $Y\in\coh\X$.
Since $\co\oplus S$ is rigid, we conclude that $X\oplus Y$ is a rigid object in $\coh\X$.
\end{proof}

\begin{lem}\label{case the slop range don't contain an integer}
Let $T$ be a tilting bundle such that the slope range of $T$ does not contain an integer, then by bundle-mutations,  $T$ can be transformed into  a tilting bundle $T'$  such that the slope range of $T'$  contains  an integer.
\end{lem}

\begin{proof}
By Lemma~\ref{exist p periods}, let $X_0$ be an indecomposable direct summand of $T$ such that  the $\tau$-period of $X_0$ is $p$. By Lemma~\ref{tilting in wing}, we may assume that $X_0$ is quasi-simple. Suppose that $\mu( X_0)=m+\frac{a_0}{b_0}$ for some integers $m,a_0,b_0$, where $(a_0,b_0)=1, 0<a_0<b_0$. By Lemma~\ref{extent integer},
 there are positive integers $a_1,b_1$ and a  quasi-simple rigid object $X_1$ in $\coh\X$ with   $\tau$-period $p$ such that   $b_0a_1-a_0b_1=1,0< a_1\leq b_1< b_0, a_1\leq a_0$, $\mu X_1=m+\frac{a_1}{b_1}$ and $X_0\oplus X_1$ is rigid in $\coh\X$.  Since $0<a_1\leq b_1<b_0$,   we can continue this process until $a_k=b_k$ for some positive integer $k$.  Denote the quasi-simple rigid objects  appeared in this process by $X_i(1\leq i\leq k)$, then we have  $\mu(X_i)=m+\frac{a_i}{b_i}$ and  $X_{i-1}\oplus X_i$ is rigid in $\coh\X$ for each  $1\leq i\leq k$.   For each $1\leq i\leq k$, we choose a tilting bundle  $T_i$  such that $X_i$ is a direct summand of $T_i$. Denote by $T_0=T$.
By Lemma~\ref{connected in case with a direct summand in perp}, $T_{i-1}$ can be transformed into  $T_i$ by bundle-mutations for each $1\leq i\leq k.$
Therefore,  $T$ can be transformed into the tilting bundle $T_k$ by bundle-mutations where the direct summand $X_k$ of $T_k$ satisfies that  $\mu(X_k)=m+\frac{a_k}{b_k}=m+1$ is an integer.
\end{proof}

\subsection{Connectedness of the subgraph}
Denote  by $\mt_{\X}$ be the set of all basic tilting objects in $\coh\X$.
By~\cite{HU}, $\mt_{\X}$   admits a partial order $\leq$.
 The {\it tilting graph} $\cg(\mt_{\X})$ of $\coh\X$ is the Hasse diagram of the poset $(\mt_\X, \leq)$. Equivalently,
the {\it tilting graph} $\mathcal{G}(\mt_\X)$ of $\coh\X$ has as vertices the isomorphism classes of basic tilting objects in $\coh\X$, while two vertices $T$ and $T'$ are connected by an edge if and only if they differ by precisely one indecomposable direct summand.
Denote  by $\mt^v_{\X}$  the set of all basic tilting bundles in $\coh\X$, $\cg(\mt^v_{\X})$ the subgraph of the tilting graph $\cg(\mt_{\X})$ consisting of all basic tilting bundles in $\coh\X$.
\begin{thm}\label{main result}
Let $\X$ be a tubular weighted projective line.
 Then the  subgraph  $\cg(\mt^v_{\X})$  of the tilting graph  is connected.
\end{thm}
\begin{proof}
 Let $T$ be a basic tilting bundle in $\coh\X$, it is enough to prove that $T$ can be obtained from $T_{can}$ by bundle-mutations.
If the slope range of $T$ contained an integer,  by Lemma
\ref{case the slop range contain an integer}, $T$ can be obtained from $T_{can}$ by bundle-mutations.

If the slope range of $T$ does not contain an integer,  by Lemma~\ref{case the slop range don't contain an integer},  $T$ can be transformed into  a tilting bundle  $T'$ by bundle-mutations such that the slope range of $T'$  contains  an integer. Then  by Lemma
\ref{case the slop range contain an integer},   $T$ can be obtained from $T_{can}$ by bundle-mutations. Therefore, $\cg(\mt^v_{\X})$  is connected.
\end{proof}

Theorem~\ref{main result} yields an alternative proof for the connectedness of the tilting graph of $\cg(\ct_\X)$.

\begin{cor}\cite{BKL}\label{cor}
Let $\X$ be a tubular weighted projective line. Then
 the  tilting graph $\cg(\mt_{\X})$ of  $\coh\X$ is connected.
\end{cor}
\begin{proof}
 By Theorem~\ref{main result}, it suffices to prove that a tilting object with direct summands of finite length can be  transformed to a tilting bundle by mutations.
Let $T$ be a tilting object in $\coh\X$ with $X_1,\cdots,X_s (s\ge 1)$ are all the indecomposable direct summands of $T$ which lie in $\mc_{\infty}$. Suppose that $q.l. X_i\leq q.l. X_{i+1}$ for $1\leq i\leq s-1$.
Let $T'=\mu_{X_1}\cdots\mu_{X_s}(T)$.
By Lemma~\ref{mutation at quasi-length}, $T'$ has no direct summand lying in $\mc_{\infty}$, \ie $T'$ a tilting bundle.
\end{proof}


\subsection*{Acknowledgments}
The author is grateful to Changjian Fu for his many helpful discussions and suggestions.

\end{document}